\documentclass[12pt]{amsart}

\title{A bound for the Waring rank of the determinant via syzygies}

\date{November 8, 2019}

\author{Mats Boij}
\address{KTH - Royal Institute of Technology, SE-100 44 Stockholm, Sweden }
\email{boij@kth.se}

\author{Zach Teitler}
\address{Department of Mathematics\\
Boise State University\\
1910 University Drive\\
Boise, ID 83725-1555, USA}
\email{zteitler@boisestate.edu}

\usepackage{amsmath,amsfonts,amssymb,amsthm}
\usepackage[T1]{fontenc} % need this to typeset "Ga\l\k{a}zka" in bibliography

\PassOptionsToPackage{obeyspaces}{url}
\usepackage[pagebackref]{hyperref}

\usepackage[margin=2.5cm]{geometry}

\usepackage{listings}
\newcommand{\bbP}{\mathbb{P}}

\newcommand{\field}{\Bbbk}

\DeclareMathOperator{\borderrank}{brk}
\DeclareMathOperator{\cactusrank}{crk}
\DeclareMathOperator{\bordercactusrank}{bcrk}

\DeclareMathOperator{\Derivs}{Derivs}
\DeclareMathOperator{\dett}{det}
\DeclareMathOperator{\per}{per}
\DeclareMathOperator{\GL}{GL}
\DeclareMathOperator{\rank}{rk}

\DeclareMathOperator{\SL}{SL}
\DeclareMathOperator{\sspan}{span}
\DeclareMathOperator{\Spec}{Spec}
\DeclareMathOperator{\Sym}{Sym}
\DeclareMathOperator{\tr}{tr}

\newcommand{\defining}[1]{\textbf{#1}}

\newcommand{\apolarityaction}{\mathbin{\circ}}
\renewcommand{\aa}{\apolarityaction}

\newtheorem{theorem}{Theorem}%[section]
\newtheorem{lemma}[theorem]{Lemma}
\newtheorem{proposition}[theorem]{Proposition}
\newtheorem{corollary}[theorem]{Corollary}

\theoremstyle{remark}
\newtheorem{remark}[theorem]{Remark}
\newtheorem*{remark*}{Remark}

\begin{document}

\begin{abstract}
We show that the Waring rank of the $3 \times 3$ determinant,
previously known to be between $14$ and $18$, is at least $15$.
We use syzygies of the apolar ideal, which have not been used in this way before.
Additionally, we show that the symmetric cactus rank of the $3 \times 3$ permanent is at least $14$.
\end{abstract}

% 15A21 - Linear and multilinear algebra; matrix theory / Basic linear algebra / Canonical forms, reductions, classification
% 15A69 - Linear and multilinear algebra; matrix theory / Basic linear algebra / Multilinear algebra, tensor products
% 14N15 - Algebraic geometry / Projective and enumerative geometry / Classical problems, Schubert calculus
% 13D02 - Commutative algebra / Homological methods / Syzygies, resolutions, complexes

\maketitle

\section{Introduction}

Let $F \in \field[x_1,\dotsc,x_n]$ be a homogeneous polynomial of degree $d$.
The \defining{Waring rank} of $F$, denoted $\rank(F)$, is the least number of terms $r$
in an expression for $F$ as a combination of $d$th powers of linear forms,
$F = c_1 \ell_1^d + \dotsb + c_r \ell_r^d$, for some $c_i \in \field$ and linear forms $\ell_i$,
or $\infty$ if no such expression exists.
For example, $xy = \frac{1}{4}(x+y)^2 - \frac{1}{4}(x-y)^2$ (as long as $\field$ has characteristic not equal to $2$),
so $\rank(xy) \leq 2$, and in fact $\rank(xy) = 2$:
if $\rank(xy) < 2$, then $xy = c \ell^2$ for some $\ell$, contradicting uniqueness of factorization.
Similarly,
\[
  xyz = \frac{1}{24} \Big\{ (x+y+z)^3 - (x+y-z)^3 - (x-y+z)^3 - (-x+y+z)^3 \Big\},
\]
showing $\rank(xyz) \leq 4$ (as long as $\field$ has characteristic not equal to $2$ or $3$);
and one can show that $\rank(xyz) = 4$.
As these examples suggest, upper bounds for Waring rank often come from explicit expressions,
while lower bounds must come from other considerations.
The Waring rank is finite for all forms of degree $d$ strictly less than the characteristic of $\field$,
or all $d$ if $\field$ has characteristic zero.

Recent introductions to Waring rank include
\cite{MR3213518},
\cite{MR2865915},
\cite{Teitler:2014gf},
and the comprehensive \cite{MR1735271}.
Interest in Waring rank arises from the intimate connection with the geometry of secant varieties
(see, for example, \cite{Alessandra-Bernardi:2018aa})
and from connections with geometric complexity theory \cite{MR3343444}.
See also very recently observed connections with parametrized algorithms \cite{Pratt:2018aa}.

Waring rank is notoriously difficult to compute.
In fact, computing Waring rank is NP-hard \cite[Conjecture 13.2]{MR3144915}, \cite[Theorem 6]{Shitov:2016aa}.
Polynomials whose Waring rank is known include
quadratic forms ($d=2$) and binary forms ($n=2$),
general polynomials
\cite{MR1311347,MR2886162},
monomials
\cite{MR2966824,MR3017012},
certain highly symmetric products of linear forms \cite{MR3306076},
reducible cubic forms \cite{MR3506475},
and elementary symmetric polynomials of odd degree \cite{MR3440150}.
Algorithms to compute Waring rank have been studied (see, for example, \cite{MR2736103})
but are not practical in most cases.

For $d \geq 1$ let $\dett_d \in \field[x_{1,1},\dotsc,x_{d,d}]$ be the generic determinant, 
%\[
%  \dett_d = \det \begin{pmatrix} x_{1,1} & \cdots & x_{1,d} \\ \vdots & & \vdots \\ x_{d,1} & \cdots & x_{d,d} \end{pmatrix}.
%\]
$\dett_d = \det(x_{i,j})_{1 \leq i,j\leq d}$.
It is well known that $\rank(\dett_2) = 4$, but no other values are known (apart from the trivial case $d=1$).
It is known that $14 \leq \rank(\dett_3) \leq 18$; in fact there have been several proofs of $\rank(\dett_3) \geq 14$,
which we review below.

Our main result is that $\rank(\dett_3) \geq 15$, which we show using syzygies of the apolar ideal.
We also show that the symmetric cactus rank of the $3 \times 3$ permanent is greater than or equal to $14$
(it was previously known to be greater than or equal to $10$).
Along the way, we also give a new proof of the known results
that the Waring rank and symmetric cactus rank of $\dett_3$, and the Waring rank of the $3 \times 3$ permanent, are greater than or equal to $14$.

\section{Background}

In this section we review background material, including the important technical tool of apolarity;
cactus rank and border rank;
conciseness and indecomposability;
and a review of previous results on Waring rank, cactus rank, and border rank of the determinant and permanent.

\subsection{Apolarity}

The primary tool in most studies of Waring rank, including this one, is apolarity.
We review the needed material here.
For a more thorough treatment see, for example, \cite{MR1735271}.

Let $S = \field[x_1,\dotsc,x_n]$ be a polynomial ring.
Denote by $V = S_1 = \sspan\{x_1,\dotsc,x_n\}$ the vector space of linear forms in $S$;
we identify $S$ with the symmetric algebra $S(V)$.
We denote by $y_1,\dotsc,y_n$ the dual basis of the dual vector space $V^*$,
and set $T = S(V^*)$, which we identify with $\field[y_1,\dotsc,y_n]$.
We let elements of $T$ act on $S$, with each $y_i$ acting as the operator $\frac{\partial}{\partial x_i}$.
We denote this action by $\aa$, as in $y_i \aa x_i^4 = 4 x_i^3$; explicitly,
\[
  y_1^{a_1} \dotsm y_n^{a_n} \aa x_1^{d_1} \dotsm x_n^{d_n} = \begin{cases}
    \prod_{i=1}^n \frac{d_i!}{(d_i-a_i)!} x_i^{d_i-a_i}, & \text{if $a_i \leq d_i$ for all $i$}, \\
    0 & \text{otherwise},
  \end{cases}
\]
extended bilinearly over sums.
This makes $S$ into a $T$-algebra, and in fact a graded $T$-algebra, although the grading as a $T$-algebra is
reversed from the standard grading as a polynomial ring:
the action takes $T_a \times S_d \to S_{d-a}$, where $S_k = 0$ if $k < 0$.

When we discuss matrix polynomials such as the determinant, which depend on the entries of a matrix,
we may use doubly-indexed variables such as $x_{1,1},\dotsc,x_{d,d}$ and corresponding dual variables $y_{1,1},\dotsc,y_{d,d}$.

Let $F \in S$ be a homogeneous form of degree $d$.
We denote by $\Derivs(F)$ the vector space spanned by $F$ and all of its derivatives of all orders,
equivalently the $T$-submodule of $S$ generated by $F$.
The \defining{apolar ideal} of $F$, or \defining{annihilating ideal}, denoted $F^\perp$,
is the ideal $\{ \theta \in T \colon \theta \aa F = 0 \}$.
It is a homogeneous ideal.
It is the kernel of the quotient map $T \to \Derivs(F) \cong T/F^\perp$, $\theta \mapsto \theta \aa F$.

An ideal $I$ is called \defining{apolar to $F$} if $I \subseteq F^\perp$.
A scheme $Z$ is called \defining{apolar to} $F$ if its defining ideal is apolar to $F$.
When $Z$ is a projective scheme, by the defining ideal we mean the saturated ideal of $Z$.

The relevance of apolarity for Waring rank is largely (perhaps entirely) due to the following well-known statement.

\begin{lemma}[Apolarity Lemma]
Let $S = \field[x_1,\dotsc,x_n]$
and $V = S_1 = \sspan\{x_1,\dotsc,x_n\}$.
Let $F \in S$ be a homogeneous polynomial of degree $d$.
Assume $\field$ has characteristic $0$ or $>d$.
Let $Z \subset \bbP V$ be a closed subscheme.
Let $\nu_d : \bbP V \to \bbP S_d$ be the $d$th Veronese map, and let $\nu_d(Z)$ be the image of $Z$ under this map.
Then $[F] \in \bbP S_d$ lies in the linear span of $\nu_d(Z)$ if and only if $Z$ is apolar to $F$.
\end{lemma}
Here the linear span of a closed subscheme of projective space is the smallest reduced linear subspace containing the scheme,
equivalently the linear subspace defined by the linear forms in the saturated ideal of the scheme.
For a proof of the Apolarity Lemma see, for example,
\cite[\textsection 4]{MR3742979}
or
\cite[\textsection 4.1]{Teitler:2014gf}.

For Waring rank, the most relevant case is when $Z$ consists of a finite, reduced set of points.
Suppose $Z = \{[\ell_1],\dotsc,[\ell_r]\} \subset \bbP S_1$; here each $\ell_i$ is a linear form.
Then $\nu_d(Z) = \{[\ell_1^d],\dotsc,[\ell_r^d]\}$, and to say that $[F]$ lies in the linear span of $\nu_d(Z)$
is precisely to say that $F = c_1 \ell_1^d + \dotsc + c_r \ell_r^d$ for some scalars $c_i$.
Thus the Waring rank of $F$ is equal to the minimum degree of a zero-dimensional reduced projective scheme apolar to $F$,
equivalently, the minimum degree of a one-dimensional radical homogeneous ideal $I$ contained in $F^\perp$.

This (reduced, zero-dimensional) version of the Apolarity Lemma may be found in many references, for example
\cite[Lemma 1.15, Theorem 5.3]{MR1735271}, \cite[\textsection 1.3]{MR1780430}, \cite[Propositions 1.3.9-10]{MR2964027}.
It goes back considerably further, to the 1851 work of Sylvester \cite{Sylvester:1851wd}, in the case of binary forms ($n=2$).

\subsection{Cactus rank and border rank}
When homogeneous polynomials are identified with symmetric tensors,
Waring rank of a homogeneous polynomial is identified with the \defining{symmetric rank} of a symmetric tensor.

The \defining{symmetric cactus rank} of a homogeneous form $F \in S_d$, denoted $\cactusrank(F)$,
is the minimum degree of a zero-dimensional, not necessarily reduced, closed subscheme $Z \subset \bbP(V)$
apolar to $F$ (so that $[F]$ lies in the span of $\nu_d(Z)$).
Equivalently, it is the minimum degree of a one-dimensional saturated, but not necessarily reduced, homogeneous
ideal $I \subset F^\perp$.
Evidently $\cactusrank(F) \leq \rank(F)$.

The \defining{symmetric border rank} of a homogeneous form $F \in S_d$, denoted $\borderrank(F)$,
is the minimum $r$ such that $F$ lies in the Zariski closure of the locus of forms of rank less than or equal to $r$.
This locus, or its projectivization, is the $r$th secant variety of the Veronese variety:
the border rank of $F$ is less than or equal to $r$ if and only if $[F]$ lies in the secant variety of $\bbP^{r-1}$s spanned by points on the Veronese variety.
Evidently $\borderrank(F) \leq \rank(F)$.

The \defining{symmetric border cactus rank} of a homogeneous form $F \in S_d$, denoted $\bordercactusrank(F)$,
is the minimum $r$ such that $F$ lies in the Zariski closure of the locus of forms of cactus rank less than or equal to $r$.
Evidently $\bordercactusrank(F) \leq \borderrank(F)$ and $\bordercactusrank(F) \leq \cactusrank(F)$.

\begin{remark}
One may similarly define the rank, cactus rank, border rank, and border cactus rank of general (non-symmetric) tensors.
In that context it is important to distinguish the rank from the symmetric rank, cactus rank from the symmetric cactus rank, etc.

Here, however, we deal exclusively with homogeneous polynomials and with the symmetric versions of these ranks.
Therefore we use the shorter terms
\emph{rank} (or Waring rank), \emph{border rank}, \emph{cactus rank}, and \emph{border cactus rank},
but we always mean the symmetric versions of these quantities.
\end{remark}

An analogue of the Apolarity Lemma for border rank has been introduced very recently, see \cite{Buczynska:2019aa}.
For several more variant notions of rank, see \cite{MR3250539}.
For examples with $\cactusrank(F_1) > \borderrank(F_1)$ and $\cactusrank(F_2) < \borderrank(F_2)$, see \cite{MR3333949}.

\subsection{Conciseness and indecomposability}

For a homogeneous form $F \in S_d = \Sym^d(V)$, the following are equivalent:
\begin{enumerate}
\item $F$ cannot be written as a polynomial in fewer variables, even after a linear change of coordinates.
More precisely, if $F \in \Sym^d(W)$ for some $W \subseteq V$, then $W=V$.
\item The projective hypersurface defined by $F=0$ is not a cone.
\item $\Derivs(F)_1 = S_1 = V$.
\item $F^\perp$ does not contain any linear forms.
\end{enumerate}
A form meeting these conditions is called \defining{concise} with respect to $V$ \cite{MR2279854}.

In general, as long as $\field$ has characteristic $0$ or $>d$,
an arbitrary degree $d$ form $F$ lies in $\Sym^d(\Derivs(F)_1)$,
and is always concise with respect to $\Derivs(F)_1$.
The \defining{essential variables} of $F$ are the elements of the subspace $\Derivs(F)_1$,
or just the elements of a basis for it.
For example, the essential variables of $xy+xz$ are $x$ and $y+z$.

A \defining{direct sum decomposition} of a homogeneous form $F \in S_d$ is
an expression
\[
  F = G(\ell_1,\dotsc,\ell_k) + H(\ell_{k+1},\dotsc,\ell_n)
\]
where $G,H$ are nonzero forms of degree $d$ and
\[
  S_1 = \sspan\{\ell_1,\dotsc,\ell_k\} \oplus \sspan\{\ell_{k+1},\dotsc,\ell_n\} ,
\]
equivalently $\ell_1,\dotsc,\ell_n \in S_1$ are linearly independent.
A form is \defining{indecomposable} (as a direct sum) if it admits no direct sum decomposition.
For example, $\dett_2 = ad-bc$ is decomposable, with the decomposition $G=ad$, $H=-bc$;
$xy$ cannot be decomposed as $xy = G(x)+H(y)$, but $xy = \frac{1}{4}(x+y)^2 - \frac{1}{4}(x-y)^2$,
so $xy$ is decomposable with $\ell_1 = x+y$, $\ell_2 = x-y$.
See \cite{MR3426613} for more on direct sum decompositions and indecomposable forms.

\subsection{Previous results on the determinant and permanent}

By Laplace expansion, $\dett_3$ and $\per_3$ are each the sum of $6$ monomials of the form $xyz$.
Expanding each monomial as a sum of powers, then, $\rank(\dett_3)$ and $\rank(\per_3)$ are each at most $6 \rank(xyz) = 24$.
For larger $d$, we have
\begin{equation}
\label{equation: naive upper bound for det and per}
  \rank(\dett_d), \rank(\per_d) \leq d! \cdot \rank(x_1 \dotsm x_d) = d! 2^{d-1} .
\end{equation}

Better upper bounds follow from identities found by Derksen \cite{MR3494510}, Krishna--Makam \cite{Krishna:2018aa},
Glynn \cite{MR2673027},
and Conner--Gesmundo--Landsberg--Ventura \cite{Conner:2019aa}.
The identity of Krishna--Makam is
\[ % Krishna-Makam's identity
\begin{split}
  \det\nolimits_3 &= x_{1,1} \, (x_{2,2}+x_{2,3}) \, (x_{3,1}+x_{3,3}) \\
    & \quad + (x_{1,2} + x_{1,3}) \, x_{2,1} \, x_{3,2} \\
    & \quad - (x_{1,1} + x_{1,3}) \, x_{2,2} \, x_{3,1} \\
    & \quad - x_{1,2} \, (x_{2,1} + x_{2,3}) \, (x_{3,2} + x_{3,3}) \\
    & \quad + (x_{1,2} - x_{1,1}) \, x_{2,3} \, (x_{3,1} + x_{3,2} + x_{3,3}).
\end{split}
\]
(Derksen had earlier found a similar identity.)
The identity implies $\rank(\dett_3) \leq 5 \rank(xyz) = 20$.
As observed by Derksen, Laplace expansion shows inductively that
\begin{equation}
\label{equation: derksen upper bound for det}
  \rank(\dett_d) \leq \left( \frac{5}{6} \right)^{\lfloor d/3 \rfloor} d! 2^{d-1} .
\end{equation}

Glynn's identity is
\[
  \per_d
  = \frac{1}{2^{d-1}} \sum_{\substack{\epsilon \in \{\pm1\}^d \\ \epsilon_1 = 1}} \prod_{i=1}^d \sum_{j=1}^d \epsilon_i \epsilon_j x_{i,j} .
\]
For example,
\[
\begin{split}
  4 \per_3 &= (x_{1,1} + x_{1,2} + x_{1,3})(x_{2,1} + x_{2,2} + x_{2,3})(x_{3,1} + x_{3,2} + x_{3,3}) \\
    & \quad - (x_{1,1} + x_{1,2} - x_{1,3})(x_{2,1} + x_{2,2} - x_{2,3})(x_{3,1} + x_{3,2} - x_{3,3}) \\
    & \quad - (x_{1,1} - x_{1,2} + x_{1,3})(x_{2,1} - x_{2,2} + x_{2,3})(x_{3,1} - x_{3,2} + x_{3,3}) \\
    & \quad + (x_{1,1} - x_{1,2} - x_{1,3})(x_{2,1} - x_{2,2} - x_{2,3})(x_{3,1} - x_{3,2} - x_{3,3}) .
\end{split}
\]
This expresses the $d \times d$ permanent as a sum of $2^{d-1}$ terms, each of rank $2^{d-1}$.
Therefore
\begin{equation}
\label{equation: glynn upper bound for per}
  \rank(\per_d) \leq 2^{2d-2} .
\end{equation}
In particular, $\rank(\per_3) \leq 16$.

In \cite{MR3492642} it was shown that $\dett_3$ cannot be written as a sum of products of linear forms with fewer than $5$ summands,
nor can $\per_3$ be written as a sum of products of linear forms with fewer than $4$ summands.
So there is no possibility of getting better upper bounds for Waring rank of $\dett_3$ or $\per_3$ by improving the above identities.
Nevertheless, Conner--Gesmundo--Landsberg--Ventura see \cite{Conner:2019aa} have very recently found
an explicit expression with $18$ terms to show $\rank(\dett_3) \leq 18$.
For brevity, we write a linear form on a space of matrices by simply writing a matrix of its coefficients,
so
\[
  \begin{pmatrix}
    a_{1,1} & \cdots  \\ \vdots
  \end{pmatrix}
  \qquad
  \text{represents the linear form}
  \qquad
  a_{1,1} x_{1,1} + \dotsb .
\]
Let $\theta = \exp(2\pi i/6)$.
Then:
\[
\begin{split}
 18\dett_3 &= \begin{pmatrix} 1 & 0 & 0 \\ 0 & -1 & 0 \\ 0 & 0 & -1 \end{pmatrix}^3 
             + \begin{pmatrix} -\theta & 0 & 0 \\ 0 & -1 & 0 \\ 0 & 0 & \theta^{-1} \end{pmatrix}^3
             + \begin{pmatrix} -\theta^{-1} & 0 & 0 \\ 0 & -1 & 0 \\ 0 & 0 & \theta \end{pmatrix}^3 \\
          &+ \begin{pmatrix} -1 & 0 & 0 \\ 0 & 0 & 1 \\ 0 & 1 & 0 \end{pmatrix}^3
             + \begin{pmatrix} -1 & 0 & 0 \\ 0 & 0 & -\theta^{-1} \\ 0 & -\theta & 0 \end{pmatrix}^3
             + \begin{pmatrix} \theta^{-1} & 0 & 0 \\ 0 & 0 & 1 \\ 0 & -\theta & 0 \end{pmatrix}^3 \\
          &+ \begin{pmatrix} 0 & -1 & 0 \\ 1 & 0 & 0 \\ 0 & 0 & 1 \end{pmatrix}^3
             + \begin{pmatrix} 0 & \theta & 0 \\ -\theta^{-1} & 0 & 0 \\ 0 & 0 & 1 \end{pmatrix}^3
             + \begin{pmatrix} 0 & \theta^{-1} & 0 \\ -\theta & 0 & 0 \\ 0 & 0 & 1 \end{pmatrix}^3 \\
          &+ \begin{pmatrix} 0 & 1 & 0 \\ 0 & 0 & -1 \\ -1 & 0 & 0 \end{pmatrix}^3
             + \begin{pmatrix} 0 & -\theta & 0 \\ 0 & 0 & \theta^{-1} \\ -1 & 0 & 0 \end{pmatrix}^3
             + \begin{pmatrix} 0 & -\theta^{-1} & 0 \\ 0 & 0 & \theta \\ -1 & 0 & 0 \end{pmatrix}^3 \\
          &+ \begin{pmatrix} 0 & 0 & 1 \\ -1 & 0 & 0 \\ 0 & -1 & 0 \end{pmatrix}^3
             + \begin{pmatrix} 0 & 0 & 1 \\ \theta^{-1} & 0 & 0 \\ 0 & \theta & 0 \end{pmatrix}^3
             + \begin{pmatrix} 0 & 0 & 1 \\ \theta & 0 & 0 \\ 0 & \theta^{-1} & 0 \end{pmatrix}^3 \\
          &+ \begin{pmatrix} 0 & 0 & -1 \\ 0 & 1 & 0 \\ 1 & 0 & 0 \end{pmatrix}^3
             + \begin{pmatrix} 0 & 0 & \theta^{-1} \\ 0 & -\theta & 0 \\ 1 & 0 & 0 \end{pmatrix}^3
             + \begin{pmatrix} 0 & 0 & \theta \\ 0 & -\theta^{-1} & 0 \\ 1 & 0 & 0 \end{pmatrix}^3 ,
\end{split}
\]
see \cite[Theorem 2.11]{Conner:2019aa}.

Conner--Gesmundo--Landsberg--Ventura also showed that $\borderrank(\dett_3) \leq 17$, see \cite[Theorem 2.12]{Conner:2019aa}.
At this time we do not know of upper bounds for the border rank of the permanent beyond
that simply the border rank is less than or equal to the Waring rank.

The rank and border rank of a general cubic form in $9$ variables are $19$ \cite{MR1311347}.
The cactus rank of a general cubic form in $9$ variables is at most $18$ \cite{MR2996880}.

As for lower bounds,
the catalecticant, or flattening, lower bound (see \cite{MR2628829}) shows
\[
  \rank(\dett_3) \geq \borderrank(\dett_3) \geq 9
\]
and
\[
  \rank(\per_3) \geq \borderrank(\per_3) \geq 9
\]
and for all $d$, for $F_d \in \{\dett_d,\per_d\}$,
\begin{equation}
\label{equation: catalecticant lower bound for det and per}
  \rank(F_d) \geq \borderrank(F_d) \geq \binom{d}{\lfloor d/2 \rfloor}^2 \approx \frac{4^d}{\pi d/2}.
\end{equation}
The catalecticant lower bound is also a lower bound for the cactus rank and border cactus rank:
$\cactusrank(\dett_3) \geq \bordercactusrank(\dett_3) \geq 9$, and so on, see \cite[Theorem 5.3D]{MR1735271}.

In \cite{MR2628829} this was improved to $\rank(\dett_3) \geq 14$, $\rank(\per_3) \geq 12$
using a lower bound for $\rank(F)$ in terms of the singularities of the hypersurface defined by $F$.
This also gave improved lower bounds for $\rank(\dett_d)$ and $\rank(\per_d)$, but the improvement was just quadratic in $d$.
Compare the exponential lower bound in \eqref{equation: catalecticant lower bound for det and per},
the exponential upper bound in \eqref{equation: glynn upper bound for per},
and the factorial upper bound in \eqref{equation: derksen upper bound for det}.

Shafiei \cite{MR3316987} used a bound of Ranestad-Schreyer \cite{MR2842085}
to show
\[
  \rank(\dett_3) \geq \cactusrank(\dett_3) \geq 10 ,
  \qquad
  \rank(\per_3) \geq \cactusrank(\per_3) \geq 10 ,
\]
and for all $d$, for $F_d \in \{\dett_d,\per_d\}$,
\[
  \rank(F_d) \geq \cactusrank(F_d) \geq \frac{1}{2}\binom{2d}{d} \approx \frac{4^d}{2\sqrt{\pi d}} .
\]
Although this gives a worse bound for $d=3$, it is asymptotically better than the bound arising from singularities.

Using a bound for the rank of forms invariant under a group action, \cite{MR3390032} showed
$\rank(\dett_3) \geq \cactusrank(\dett_3) \geq 14$,
and
\[
  \rank(\dett_d) \geq \cactusrank(\dett_d) \geq \binom{2d}{d} - \binom{2d-2}{d-1} \approx \frac{3}{4} \binom{2d}{d} \approx \frac{3 \cdot 4^d}{4 \sqrt{\pi d}} .
\]
(The bound for invariant forms does not seem to give any interesting results for the permanent, because its stabilizer is too small.)

Farnsworth \cite{MR3427655} used Koszul-Young flattenings to show
\[
  \rank(\dett_3) \geq \borderrank(\dett_3) \geq 14,
  \qquad
  \rank(\per_3) \geq \borderrank(\per_3) \geq 14,
\]
and $\borderrank(\dett_4) \geq 38$,
along with modestly improved lower bounds for $\borderrank(\dett_d)$ and $\borderrank(\per_d)$ when $d \geq 5$.

Ga\l\k{a}zka \cite{MR3611482} showed that certain lower bounds for border rank are in fact also lower bounds for cactus rank and border cactus rank.
%This applies at least to Farnsworth's lower bounds for $\borderrank(\dett_d)$ and $\borderrank(\per_d)$ for $d \geq 5$.
%It is not immediately clear whether Ga\l\k{a}zka's result applies also to Farnsworth's lower bound for
%$\borderrank(\dett_3)$, $\borderrank(\per_3)$, and $\borderrank(\dett_4)$, which are obtained by a slightly different method.
Specifically, Ga\l\k{a}zka shows that lower bounds for border rank arising from Young flattenings
are in fact lower bounds for cactus rank and border cactus rank, if the Young flattenings arise from vector bundles.
For the $d \geq 5$ cases, Farnsworth uses Young flattenings that indeed arise from vector bundles.
It is not immediately clear whether the Young flattenings used by Farnsworth for the $d=3,4$ cases also arise from vector bundles.
If so, then combining Farnsworth's and Ga\l\k{a}zka's results gives $\cactusrank(\per_3) \geq 14$.

In summary, three (at least) proofs of $\rank(\dett_3) \geq 14$ have been given
\cite{MR2628829,MR3390032,MR3427655}.
The chronologically first proof was strictly for Waring rank;
the second proof gave a bound for cactus rank, in addition to Waring rank;
and the third proof gave a bound for border rank, in addition to Waring rank.
We are aware of only one proof of $\rank(\per_3) \geq 14$, in \cite{MR3427655}.
That proof gave a bound for border rank, in addition to Waring rank.
It seems that up to now, the best bound for the cactus rank of the $3 \times 3$ permanent is $\cactusrank(\per_3) \geq 10$ \cite{MR3316987},
unless Farnsworth's and Ga\l\k{a}zka's results may be combined to yield $\cactusrank(\per_3) \geq 14$.

\section{Bound for Waring rank via syzygies}

In order to illustrate the method we will use for determinant and permanent,
we illustrate briefly with a proof of $\rank(xyz) \geq 4$ (which is well-known by other means).
The explicit expression given in the introduction shows $\rank(xyz) \leq 4$.
Suppose that $xyz = \ell_1^3 + \dotsb + \ell_r^3$.
Note that $xyz$ is concise, since each variable can be obtained as a derivative of $xyz$:
$x = \frac{\partial^2}{\partial y \, \partial z}(xyz)$, and so on.
This means that the $\ell_i$ must span $V = \sspan\{x,y,z\}$, so already we have $r \geq 3$.
If $r < 4$, then the $\ell_i$ must give a basis for $V$.
Up to linear change of coordinates, $I = I(\{[\ell_1],[\ell_2],[\ell_3]\})$ is given by $I = (tu,tv,uv)$,
where $t,u,v$ are suitable coordinates.
The minimal graded free resolution of $T/I$ has graded Betti numbers
\[
\begin{array}{r rrr}
  &0&1&2\\
  \text{total:}&1&3&2\\
  \text{0:} & 1 &  . &   . \\
  \text{1:} & . &  3 &   2 \\
\end{array}
\]
Now we compare the apolar ideal.
Here it is more convenient to replace $xyz$ with $x_1 x_2 x_3$, with dual variables $y_1,y_2,y_3$.
We have $(x_1 x_2 x_3)^\perp = (y_1^2, y_2^2, y_3^2)$, whose resolution is Koszul, with graded Betti numbers
\[
\begin{array}{r rrrr}
  &0&1&2&3\\
  \text{total:}&1&3&3&1\\
  \text{0:} & 1 &  . &   . &  . \\
  \text{1:} & . &  3 &   . &  . \\
  \text{2:} & . &  . &   3 &  . \\
  \text{3:} & . &  . &   . &  1 \\
\end{array}
\]
In particular, $\beta_{2,3}(T/I) = 2$, while $\beta_{2,3}(T/(xyz)^\perp) = 0$.

In general, an inclusion $I \subseteq J$ does not imply $\beta_{i,j}(T/I) \leq \beta_{i,j}(T/J)$.
However, we do get an inequality in the lowest-degree strand of the resolution of the larger ideal,
that is, along the first non-trivial row of the Betti table.
For our purposes the following special case is sufficient:
\begin{proposition}
\label{proposition: subideal betti number inequality}
Suppose that $I \subseteq J$ are homogeneous ideals in the polynomial ring $T$ with standard grading, and that $J$ contains no linear forms.
Then $\beta_{i,i+1}(T/I) \leq \beta_{i,i+1}(T/J)$ for all $i$.
\end{proposition}
See for example \cite[Proposition 8.11]{MR2103875} for a more general statement.

Now the inclusion $I \subset (xyz)^\perp$, and the fact that $(xyz)^\perp$ contains no linear form,
means that $\beta_{2,3}(T/I) \leq \beta_{2,3}(T/(xyz)^\perp)$.
This contradiction shows $\rank(xyz) > 3$.

We take a similar approach for the determinant and permanent.
One can compute that $\beta_{5,6}(T/\dett_3^\perp) = 100$ and $\beta_{5,6}(T/\per_3^\perp) = 116$.
Then we show that for suitable ideals $I$ of $13$ points, $\beta_{5,6}(T/I) \geq 140$.
This already proves $\rank(\dett_3), \rank(\per_3) > 13$, and in fact it gives a lower bound for cactus rank.
Finally we analyze $\beta_{5,6}(T/(\dett_3 - \ell^3)^\perp)$, and show that this is still strictly less than $140$.
This proves $\rank(\dett_3 - \ell^3) > 13$, and so therefore $\rank(\dett_3) > 14$.

\subsection{Apolarity of determinant and permanent}

Let $X$ be the matrix $(x_{i,j})$ and $Y$ be the matrix $(y_{i,j})$ of dual variables,
so $\dett_d = \det(X)$ and $\per_d = \per(X)$.
Shafiei determined the apolar ideals of the determinant and permanent, as follows.
\begin{theorem}[\cite{MR3316987}]
For $d \geq 2$, $\dett_d^\perp$ and $\per_d^\perp$ are generated by quadrics.
Specifically, $\dett_d^\perp$ is generated by the following quadrics:
\begin{enumerate}
\item $y_{i,j}^2$ for $1 \leq i,j \leq d$ (squares of entries of $Y$),
\item $y_{i,j_1} y_{i,j_2}$ for $1 \leq i,j_1,j_2 \leq d$ (products of two entries from the same row of $Y$),
\item $y_{i_1,j} y_{i_2,j}$ for $1 \leq i_1,i_2,j \leq d$ (products of two entries from the same column of $Y$), and
\item $y_{i,j} y_{k,l} + y_{i,l} y_{k,j}$ for $1 \leq i,j,k,l \leq d$ (permanents of $2 \times 2$ submatrices of $Y$).
\end{enumerate}
And $\per_d^\perp$ is generated by
\begin{enumerate}
\item $y_{i,j}^2$ for $1 \leq i,j \leq d$,
\item $y_{i,j_1} y_{i,j_2}$ for $1 \leq i,j_1,j_2 \leq d$,
\item $y_{i_1,j} y_{i_2,j}$ for $1 \leq i_1,i_2,j \leq d$, and
\item $y_{i,j} y_{k,l} - y_{i,l} y_{k,j}$ for $1 \leq i,j,k,l \leq d$ (determinants of $2 \times 2$ submatrices of $Y$).
\end{enumerate}
\end{theorem}
The ideals described in the theorem are easily shown to be contained in the apolar ideals of $\dett_d$ and $\per_d$,
and then Shafiei in \cite{MR3316987} shows that they have the same Hilbert function.

Alternatively, if $\Theta$ is a homogeneous form in $\dett_d^\perp$ or $\per_d^\perp$,
one may use the monomial generators to remove all terms of $\Theta$ that involve two entries from the same row or column of $Y$,
then use the binomial generators to eliminate all terms that involve any southwest-northeast entries,
i.e., any terms with factors $y_{i,j} y_{k,l}$ with $i > k$ and $j < l$.
So we can assume $\Theta$ consists purely of northwest-southeast terms, i.e., terms of the form
$y_{i_1,j_1} \dotsm y_{i_t,j_t}$ with $i_1 < \dotsb < i_t$ and $j_1 < \dotsb < j_t$.
But such forms can never annihilate $\dett_d$ or $\per_d$, since the differentiations
result in combinations of determinants (respectively, permanents) of distinct submatrices,
which are linearly independent.

We are especially interested in $\dett_3$ and $\per_3$,
and not only the generators, but also the syzygies of their apolar ideals.
The graded Betti numbers of $T/\dett_3^\perp$ are as follows:
\[
\begin{array}{rrrrrrrrrrr}
  &0&1&2&3&4&5&6&7&8&9\\
  \text{total:}&1&36&160&315&388&388&315&160&36&1\\
  \text{0:}&1&\text{.}&\text{.}&\text{.}&\text{.}&\text{.}&\text{.}&\text{.}&\text{.}&\text{.}\\
  \text{1:}&\text{.}&36&160&315&288&100&\text{.}&\text{.}&\text{.}&\text{.}\\
  \text{2:}&\text{.}&\text{.}&\text{.}&\text{.}&100&288&315&160&36&\text{.}\\
  \text{3:}&\text{.}&\text{.}&\text{.}&\text{.}&\text{.}&\text{.}&\text{.}&\text{.}&\text{.}&1\\
\end{array}
\]
The graded Betti numbers of $T/\per_3^\perp$ are as follows:
\[
\begin{array}{rrrrrrrrrrr}
  &0&1&2&3&4&5&6&7&8&9\\
  \text{total:}&1&36&160&315&404&404&315&160&36&1\\
  \text{0:}&1&\text{.}&\text{.}&\text{.}&\text{.}&\text{.}&\text{.}&\text{.}&\text{.}&\text{.}\\
  \text{1:}&\text{.}&36&160&315&288&116&\text{.}&\text{.}&\text{.}&\text{.}\\
  \text{2:}&\text{.}&\text{.}&\text{.}&\text{.}&116&288&315&160&36&\text{.}\\
  \text{3:}&\text{.}&\text{.}&\text{.}&\text{.}&\text{.}&\text{.}&\text{.}&\text{.}&\text{.}&1\\
\end{array}
\]
These are easy to compute in Macaulay2 \cite{M2}.
For our purposes, the important point is that $\beta_{5,6}(T/\dett_3^\perp) = 100$ and $\beta_{5,6}(T/\per_3^\perp) = 116$.

\begin{remark}
The syzygies of $\dett_d^\perp$ and $\per_d^\perp$
have been studied by Alper and Rowlands \cite{Alper:2017aa},
who give combinatorial interpretations of some syzygies.
\end{remark}

\subsection{Ideals of points}

\begin{proposition}
\label{prop: points betti bound 140}
Let $I \subset S = \field[y_1,\dotsc,y_9]$ be a one-dimensional saturated homogeneous ideal of degree $13$ containing no linear form.
Then $\beta_{5,6}(T/I) \geq 140$.
\end{proposition}

\begin{proof}
Peeva \cite{MR2084070} showed that the Betti numbers of $T/I$
are obtained by consecutive cancellations from those of $T/L$ where $L$ is the lex-segment ideal with the same $h$-vector as $I$.
So $\beta_{5,6}(T/I) \geq \beta_{5,6}(T/L) - \beta_{4,6}(T/L)$.
(A priori subtracting $\beta_{6,6}(T/L)$ would also be possible,
but the condition that $I$ contains no linear form means that neither does $L$, so $\beta_{6,6}(T/L) = 0$.)

The $h$-vector of $T/I$ must start with $(1,8,\dotsc)$ since $T/I$ contains no linear form.
The entries of the $h$-vector sum to $13$, the degree of $T/I$.
After the $8$, the entries of the $h$-vector are nonincreasing,
by the Macaulay bounds, see \cite[Theorem 4.2.10]{MR1251956}.
Explicitly, if the $h$-vector's $d$th entry, counted from $0$, is less than or equal to $d$,
then the Macaulay bounds imply that from that point on the $h$-vector is nonincreasing.
So if the $h$-vector of $T/I$ starts with $(1,8,2,\dotsc)$ or $(1,8,1,\dotsc)$ then the remainder is nonincreasing.
If it starts with $(1,8,3,\dotsc)$ then because the entries add up to $13$ the vector must be $(1,8,3,1)$.
Hence the $h$-vector of $T/I$ must be one of
$(1,8,4)$, $(1,8,3,1)$, $(1,8,2,2)$, $(1,8,2,1,1)$ or $(1,8,1,1,1,1)$.

The Betti numbers for the lex-segment ideals with these $h$-vectors are given as follows.
For $h$-vector $(1,8,4)$:
\[
  \begin{array}{rrrrrrrrrr}
      &0&1&2&3&4&5&6&7&8\\
      \text{}&1&37&175&399&539&455&237&70&9\\
      \text{0:}&1&\text{.}&\text{.}&\text{.}&\text{.}&\text{.}&\text{.}&\text{.}&\text{.}\\
      \text{1:}&\text{.}&32&141&300&379&300&147&41&5\\
      \text{2:}&\text{.}&5&34&99&160&155&90&29&4\\
  \end{array}
\]
For $h$-vector $(1,8,3,1)$:
\[
  \begin{array}{rrrrrrrrrr}
         &0&1&2&3&4&5&6&7&8\\
         \text{}&1&37&175&399&539&455&237&70&9\\
         \text{0:}&1&\text{.}&\text{.}&\text{.}&\text{.}&\text{.}&\text{.}&\text{.}&\text{.}\\
         \text{1:}&\text{.}&33&148&321&414&335&168&48&6\\
         \text{2:}&\text{.}&3&20&57&90&85&48&15&2\\
         \text{3:}&\text{.}&1&7&21&35&35&21&7&1\\
  \end{array}
\]
For $h$-vector $(1,8,2,2)$:
\[
    \begin{array}{rrrrrrrrrr}
      &0&1&2&3&4&5&6&7&8\\
      \text{}&1&36&168&378&504&420&216&63&8\\
      \text{0:}&1&\text{.}&\text{.}&\text{.}&\text{.}&\text{.}&\text{.}&\text{.}&\text{.}\\
      \text{1:}&\text{.}&34&154&336&434&350&174&49&6\\
      \text{2:}&\text{.}&\text{.}&\text{.}&\text{.}&\text{.}&\text{.}&\text{.}&\text{.}&\text{.}\\
      \text{3:}&\text{.}&2&14&42&70&70&42&14&2
    \end{array}
\]
For $h$-vector $(1,8,2,1,1)$:
\[
    \begin{array}{rrrrrrrrrr}
      &0&1&2&3&4&5&6&7&8\\
      \text{}&1&36&168&378&504&420&216&63&8\\
      \text{0:}&1&\text{.}&\text{.}&\text{.}&\text{.}&\text{.}&\text{.}&\text{.}&\text{.}\\
      \text{1:}&\text{.}&34&154&336&434&350&174&49&6\\
      \text{2:}&\text{.}&1&7&21&35&35&21&7&1\\
      \text{3:}&\text{.}&\text{.}&\text{.}&\text{.}&\text{.}&\text{.}&\text{.}&\text{.}&\text{.}\\
      \text{4:}&\text{.}&1&7&21&35&35&21&7&1\\
    \end{array}
\]
For $h$-vector $(1,8,1,1,1,1)$:
\[
    \begin{array}{rrrrrrrrrr}
      &0&1&2&3&4&5&6&7&8\\
      \text{}&1&36&168&378&504&420&216&63&8\\
      \text{0:}&1&\text{.}&\text{.}&\text{.}&\text{.}&\text{.}&\text{.}&\text{.}&\text{.}\\
      \text{1:}&\text{.}&35&161&357&469&385&195&56&7\\
      \text{2:}&\text{.}&\text{.}&\text{.}&\text{.}&\text{.}&\text{.}&\text{.}&\text{.}&\text{.}\\
      \text{3:}&\text{.}&\text{.}&\text{.}&\text{.}&\text{.}&\text{.}&\text{.}&\text{.}&\text{.}\\
      \text{4:}&\text{.}&\text{.}&\text{.}&\text{.}&\text{.}&\text{.}&\text{.}&\text{.}&\text{.}\\
      \text{5:}&\text{.}&1&7&21&35&35&21&7&1\\
    \end{array}
\]
These can be produced in Macaulay2 with code such as the following.
\begin{lstlisting}
lexSeg = (R,d,i) -> ideal ((basis(d,R))_{0..(numcols basis(d,R)-i-1)})
lexIdeal = L -> (
  R := QQ[q_1..q_(L_1)]; trim sum for i to #L-1 list lexSeg(R,i,L_i));

betti res lexIdeal(1,8,4,0)
betti res lexIdeal(1,8,3,1,0)
\end{lstlisting}
and so on.
(To be precise, this code produces the Artinian reductions of the desired one-dimensional ideals.
But they have the same graded Betti numbers.)

This shows that in all cases, after taking into account the
possible cancellations, the Betti number $\beta_{5,6}(T/I) \geq 140$.
\end{proof}

In the following theorem,
the only new result is that $\cactusrank(\per_3) \geq 14$, improving the previously known $\cactusrank(\per_3) \geq 10$ \cite{MR3316987}.
\begin{theorem}
\label{theorem: lower bound 14}
$\rank(\dett_3) \geq \cactusrank(\dett_3) \geq 14$ and $\rank(\per_3) \geq \cactusrank(\per_3) \geq 14$.
\end{theorem}
\begin{proof}
Let $I$ be any one-dimensional saturated homogeneous ideal which is apolar to $\dett_3$ or $\per_3$.
Since $\dett_3$ and $\per_3$ are concise, $\dett_3^\perp$ and $\per_3^\perp$ contain no linear forms, hence neither does $I$.
If $I$ has degree $13$ then $\beta_{5,6}(T/I) \geq 140$ by Proposition~\ref{prop: points betti bound 140}.
However, by
Proposition~\ref{proposition: subideal betti number inequality},
$\beta_{5,6}(T/I) \leq \beta_{5,6}(T/\dett_3^\perp) = 100$
or $\beta_{5,6}(T/I) \leq \beta_{5,6}(T/\per_3^\perp) = 116$, a contradiction.
\end{proof}

\begin{remark}
The Betti table of the minimal graded free resolution of the coordinate ring of a general set of $13$ points in $\bbP^8$ is as follows:
\[
\begin{array}{r rrrrr rrrr}
      &0&1&2&3&4&5&6&7&8\\
      \text{}&1&36&168&378&504&420&216&63&8\\
  \text{0:} & 1 &  . &   . &   . &   . &   . &   . &   . &  . \\
  \text{1:} & . & 32 & 136 & 266 & 280 & 140 &   2 &   . &  . \\
  \text{2:} & . &  . &   . &   . &   . &  10 &  49 &  24 &  4
\end{array}
\]
See \cite{MR1688433} or \cite[\textsection 5]{MR1749894}.
This implies, by semicontinuity of graded Betti numbers, that the graded Betti number $\beta_{5,6}$ is at least $140$
for an arbitrary set of $13$ points in $\bbP^8$, not lying on a hyperplane.
So, why not just give this argument; why the longer argument given above?

In fact, this shorter argument implies the inequality
for any closed, nondegenerate zero-dimensional subscheme of $\bbP^8$
which can be deformed into a set of $13$ distinct points---a \emph{smoothable} scheme.
This would be sufficient if we only cared about ideals of reduced schemes, corresponding to Waring rank,
or smoothable schemes, corresponding to so-called smoothable rank.

But there exist non-smoothable schemes, see for example
\cite{MR515043,MR3689951}.
(These give non-smoothable schemes of degree less than $13$, in spaces of lower dimension than $\bbP^8$.
But the degree can be increased by adding disjoint points, and the embedding dimension can be increased by re-embedding.)
So this shorter argument is not sufficient to deal directly with arbitrary schemes.

On the other hand, it is known that every scheme of minimal degree that spans a given form is Gorenstein
(if $R$ is a zero-dimensional scheme such that $[F]$ lies in the linear span of $R$ and $R$ is not Gorenstein,
then there is a subscheme $R' \subset R$ of strictly lower degree and whose linear span still includes $[F]$, see \cite[Lemma 2.3]{MR3121848}).
Thus we can restrict attention to Gorenstein schemes.
And it is known that every Gorenstein scheme of degree at most $13$ in $\bbP^8$ is smoothable \cite{MR3404648}
(under some mild assumptions on the field).
Combining these results gives an alternative proof of Theorem~\ref{theorem: lower bound 14}.
\end{remark}

\begin{remark}
A similar computation shows that a one-dimensional ideal $I$ of degree $14$ in $\mathbb{P}^8$ containing no linear form has
$\beta_{5,6}(T/I) \geq 70$, and $70$ does occur (in particular for the ideal of a general set of $14$ reduced points).
This lower bound is far too weak to use in our argument to prove $\rank(T/\dett_3) > 14$.
Instead we will use the bounds for ideals of degree $13$.
\end{remark}

\subsection{Conciseness after subtracting}

\begin{lemma}
Let $F$ be any concise, indecomposable form and let $\ell$ be any linear form.
Then $F - \ell^d$ is concise.
\end{lemma}
\begin{proof}
If $F - \ell^d$ is not concise, then either $\ell$ is an essential variable of $F-\ell^d$,
in which case $F = \ell^d + (F-\ell^d)$ is not concise,
or else $\ell$ is independent of the essential variables of $F-\ell^d$,
in which case $F = \ell^d + (F-\ell^d)$ is a direct sum decomposition of $F$.
\end{proof}

It is shown in \cite{MR3426613}
that $\dett_d$ and $\per_d$ are indecomposable for $d \geq 3$.
And it is easy to check that $\dett_d$ and $\per_d$ are concise.

\begin{corollary}
\label{corollary: concise}
For any linear form $\ell$, $\dett_d - \ell^d$ and $\per_d - \ell^d$ are concise.
\end{corollary}

\subsection{Apolarity after subtracting}

Now we consider the apolarity of $\dett_3 - \ell^3$ for a linear form $\ell$.
Our goal is to show $\beta_{5,6}(T/(\dett_3 - \ell^3)^\perp) < 140$ for every linear form $\ell$.

\begin{proposition}
Fix $d > 0$.
Assume that $\field$ is closed under taking $d$th roots.
Let $X$ be the generic $d \times d$ matrix $X = (x_{i,j})$, $\dett_d = \det X$,
$S = \field[X] = \field[x_{1,1},\dotsc,x_{d,d}]$.
Let $\ell \in S$ be an arbitrary nonzero linear form.
There is a linear change of coordinates leaving $\dett_d$ invariant and taking $\ell^d$ to
$\lambda (x_{1,1} + \dotsb + x_{k,k})^d$ for some $1 \leq k \leq d$ and some scalar $\lambda\neq0$.
In addition, if $1 \leq k \leq d-1$, then we can take $\lambda = 1$.
\end{proposition}

\begin{proof}
Let $s_1,s_2 \in \SL_d$.
Our linear change of coordinates in $S$ will be the substitution of $X$ with the matrix product $s_1 X s_2$.
We have $\dett_d = \det(X) = \det(s_1 X s_2)$, which is to say that for any $s_1$ and $s_2$, this coordinate change leaves $\dett_d$ invariant.

Write $\ell = \sum_{i,j} a_{i,j} x_{i,j}$.
Let $A = (a_{i,j})$ be the $d \times d$ matrix of coefficients of $\ell$.
Then $\ell = \tr(A X^t)$.
The coordinate change takes $\ell$ to
\[
  \tr(A (s_1 X s_2)^t) = \tr(A s_2^t X^t s_1^t) = \tr(s_1^t A s_2^t X^t) = \tr(A' X^t),
\]
where $A' = s_1^t A s_2^t$.
We can choose row operations and column operations to make $A$ diagonal,
and swap rows and columns so that the first $k = \rank(A)$ diagonal entries of $A$ are nonzero and the rest are zero.
Multiply the $i$th row of $A$ by $a_{i,i}^{-1}$ and the $d$th row of $A$ by $a_{i,i}$, for $i$ from $1$ to $k$ or $d-1$, whichever is less.
If $k < d$, then we have reached $A'$ with diagonal entries consisting of $k$ ones followed by $d-k$ zeros, as desired.

If $k = d$, then at this point our matrix has diagonal entries consisting of $d-1$ entries of $1$ followed by $\det(A)$, which is nonzero.
Multiply each of the first $d-1$ rows by a $d$th root of $\det(A)$ and the last row by the reciprocal
to ensure that all the diagonal entries of $A'$ are equal,
i.e., $A'$ is a scalar multiple of the identity, namely $A' = cI$, $c = \det(A)^{1/d}$.
Then $\ell^d = (c x_{1,1} + \dotsb + c x_{d,d})^d = \lambda (x_{1,1} + \dotsb + x_{d,d})^d$, $\lambda = c^d = \det(A)$.
\end{proof}

This means that when we consider $\dett_3 - \ell^d$, we can reduce to the three following cases:
\begin{enumerate}
  \item $\dett_3 - x_{1,1}^3$,
  \item $\dett_3 - (x_{1,1} + x_{2,2})^3$,
  \item $\dett_3 - \lambda (x_{1,1} + x_{2,2} + x_{3,3})^3$ for a nonzero scalar $\lambda$.
\end{enumerate}

\begin{proposition}
$\beta_{5,6}(T/(\dett_3 - x_{1,1}^3)^\perp) = \beta_{5,6}(T/(\dett_3 - (x_{1,1} + x_{2,2})^3)^\perp) = 100$.
\end{proposition}
\begin{proof}
A direct computation in Macaulay2.
\end{proof}

We will next show that $\beta_{5,6}(T/(\dett_3 - \lambda(x_{1,1} + x_{2,2} + x_{3,3})^3)^\perp) < 140$ for all $\lambda \in \field$.

It is convenient to simplify notation by replacing $x_{1,1},\dotsc,x_{3,3}$ with $x_1,\dotsc,x_9$,
with dual variables $y_1,\dotsc,y_9$,
where
\[
  X = \begin{pmatrix}
    x_{1,1} & x_{1,2} & x_{1,3} \\
    x_{2,1} & x_{2,2} & x_{2,3} \\
    x_{3,1} & x_{3,2} & x_{3,3}
  \end{pmatrix}
  =
  \begin{pmatrix}
    x_1 & x_2 & x_3 \\
    x_4 & x_5 & x_6 \\
    x_7 & x_8 & x_9
  \end{pmatrix},
  \qquad
  Y = \begin{pmatrix}
    y_{1,1} & y_{1,2} & y_{1,3} \\
    y_{2,1} & y_{2,2} & y_{2,3} \\
    y_{3,1} & y_{3,2} & y_{3,3}
  \end{pmatrix}
  =
  \begin{pmatrix}
    y_1 & y_2 & y_3 \\
    y_4 & y_5 & y_6 \\
    y_7 & y_8 & y_9
  \end{pmatrix}.
\]
It is also convenient for computations to introduce $\mu$ and homogenize as in the following statement.

\begin{proposition}
Let $\mu,\lambda \in \field$, $\mu \neq 0$,
and let $F = F(\mu,\lambda) = \mu \dett_3 - \lambda (x_1 + x_5 + x_9)^3$.
Then $F^\perp$ is generated by the following $36$ linearly independent quadrics.
\[
\begin{split}
  F^\perp = (&
      {y}_{2}^{2},
      {y}_{3}^{2},
      {y}_{4}^{2},
      {y}_{6}^{2},
      {y}_{7}^{2},
      {y}_{8}^{2},
      {y}_{1}^{2}-{y}_{9}^{2},
      {y}_{5}^{2}-{y}_{9}^{2},\\
      &
      {y}_{1} {y}_{2},
      {y}_{1} {y}_{3},
      {y}_{1} {y}_{4},
      {y}_{1} {y}_{7},
      {y}_{2} {y}_{3},
      {y}_{2} {y}_{5},
      {y}_{2} {y}_{8},
      {y}_{3} {y}_{6},
      {y}_{3} {y}_{9}, \\
      &
      {y}_{4} {y}_{5},
      {y}_{4} {y}_{6},
      {y}_{4} {y}_{7},
      {y}_{5} {y}_{6},
      {y}_{5} {y}_{8},
      {y}_{6} {y}_{9},
      {y}_{7} {y}_{8},
      {y}_{7} {y}_{9},
      {y}_{8} {y}_{9}, \\
      &
      {y}_{1} {y}_{6} + {y}_{3} {y}_{4},
      {y}_{1} {y}_{8} + {y}_{2} {y}_{7},
      {y}_{2} {y}_{9} + {y}_{3} {y}_{8}, \\
      &
      {y}_{2} {y}_{6} + {y}_{3} {y}_{5},
      {y}_{4} {y}_{8} + {y}_{5} {y}_{7},
      {y}_{4} {y}_{9} + {y}_{6} {y}_{7}, \\
      &
      {y}_{1} {y}_{5} + {y}_{2} {y}_{4} - {y}_{9}^{2},
      {y}_{1} {y}_{9} + {y}_{3} {y}_{7} - {y}_{9}^{2},
      {y}_{5} {y}_{9} + {y}_{6} {y}_{8} - {y}_{9}^{2}, \\
      &
       \mu {y}_{1}^{2}-6 \lambda ( {y}_{6} {y}_{8} + {y}_{3} {y}_{7} + {y}_{2} {y}_{4} ) \; ).
\end{split}
\]
\end{proposition}
The above generators were found with the assistance of Macaulay2.
\begin{proof}
It is easy to check that all of the listed generators annihilate $F$.
Let $H$ be the ideal generated by the $36$ listed generators on the right hand side,
so $H \subseteq F^\perp$.

By Corollary~\ref{corollary: concise}, $F$ is concise, so $F^\perp$ contains no linear form.
Therefore the Hilbert function of $F^\perp$ begins with $1,9,\dotsc$.
By Gorenstein symmetry, the full Hilbert function is $1,9,9,1$.
Therefore $(F^\perp)_2$ has codimension $9$ in the space of quadrics in $T = \field[y_1,\dotsc,y_9]$.
That space has dimension $\binom{8+2}{2} = 45$.
So $F^\perp$ contains $36$ linearly independent quadrics.

The $36$ quadrics generating $H$ are linearly independent because they have distinct leading monomials.
So $H$ agrees with $F^\perp$ in degree $2$ (as well as degrees less than $2$).

Since $F$ has degree $3$, $F^\perp$ contains all forms of degree $4$,
so $F^\perp$ certainly does not have any generator of degree greater than $4$.

And $F^\perp$ does not have a generator of degree equal to $4$.
Indeed,
\cite[Proposition 1.6]{MR3426613}
asserts that
if $G$ is a form of degree $d$ such that $G^\perp$ has a minimal generator of degree $d+1$,
then $G$ must be a power of a linear form.
But $F$ is concise in $9$ variables by Corollary~\ref{corollary: concise},
so $F$ is not a power of a linear form (such powers are concise with respect to $1$ variable).

Finally we eliminate the possibility that $F^\perp$ has any generators in degree $3$.
One computes directly that the first $35$ generators of $H$ generate a codimension $2$ space of cubics,
while $(F^\perp)_3$ has codimension $1$.
The element $\mu y_1^3 - 6 \lambda y_1 y_6 y_8$ is in $H$ (using the $36$th generator),
but not in the subideal generated by the first $35$ quadrics,
because each of those first $35$ quadrics, considered as a polynomial on matrices, vanishes on the identity matrix,
while $\mu y_1^3 - 6 \lambda y_1 y_6 y_8$ does not.
This shows that $H_3$ has codimension $1$, so $H$ coincides with $F^\perp$ in degree $3$.
Therefore the cubics in $F^\perp$ are generated by the quadrics, and thus the generators of $F^\perp$ are the quadrics.
\end{proof}

\begin{proposition}
With notation as above, $\beta_{5,6}(T/F^\perp) < 140$.
\end{proposition}
\begin{proof}
We may assume the field $\field$ is algebraically closed, since the graded Betti numbers do not change under field extension.
Let $\tilde{T} = T[\mu,\lambda] = \field[y_1,\dotsc,y_9,\mu,\lambda]$, where $y_1,\dotsc,y_9$ have degree $1$ and $\mu,\lambda$ have degree $0$.
Let $\tilde{H} \subset \tilde{T}$
be the ideal generated by the quadrics listed above, where $\mu,\lambda$ are read as variables instead of scalars.
(This is not the apolar ideal of ``$\tilde{F}$''; that apolar ideal would include $\mu^2, \mu\lambda, \lambda^2$.)
One computes in Macaulay2 a graded free resolution $\tilde{C}$ of $\tilde{T}/\tilde{H}$
with graded Betti numbers as follows.
\[
\begin{array}{r rrrrr rrrrr}
  &0&1&2&3&4&5&6&7&8&9\\
  \text{total:}&1&36&160&316&424&479&399&196&45&1\\
  \text{0:} & 1 &  . &   . &   . &   . &   . &   . &   . &  . & . \\
  \text{1:} & . & 36 & 160 & 315 & 289 & 135 &  56 &  28 &  8 & 1 \\
  \text{2:} & . &  . &   . &   1 & 135 & 344 & 343 & 168 & 37 & . \\
  \text{3:} & . &  . &   . &   . &   . &   . &   . &   . &  . & 1
\end{array}
\]
In particular, $\beta_{5,6}(\tilde{T}/\tilde{H}) = 135 < 140$.
Substituting any values of $\mu,\lambda \in \field$, $\mu \neq 0$, takes $\tilde{H}$ to $H = F^\perp$
and takes $\tilde{C}$ to a complex $C$ that resolves $T/H$.

We claim that this complex $C$ is exact.
For all values $\mu,\lambda \in \field$, $\mu \neq 0$, $F = F(\mu,\lambda)$ is a concise cubic form in $9$ variables.
So the Hilbert function of $T/F^\perp$ starts with $(1,9,\dotsc)$.
By Gorenstein symmetry it is $(1,9,9,1)$.
This is constant with respect to $\mu,\lambda$, so the family of algebras $T/F(\mu,\lambda)^\perp$
has the same Hilbert polynomial (the constant $20$) at every closed point of $\Spec \field[\mu^{\pm 1},\lambda]$.
We can choose a constant monomial cobasis for the family of ideals $F(\mu,\lambda)^\perp$
(i.e., a set of monomials whose images give a basis for the algebras $T/F(\mu,\lambda)^\perp$),
so that the family $T/F(\mu,\lambda)^\perp$ is locally free over $\field[\mu^{\pm 1},\lambda]$.
Therefore the Hilbert polynomial of the family $T/F(\mu,\lambda)^\perp$ is constant (at all points, not just closed points),
so the family is flat over $\field[\mu^{\pm 1},\lambda]$ \cite[Theorem III.9.9]{hartshorne}.
By induction on homological degree, the kernels of maps in $\tilde{C}$ are flat.
Substituting $\mu_0,\lambda_0$ for $\mu,\lambda$ is tensoring with $\field[\mu^{\pm 1},\lambda]/(\mu-\mu_0,\lambda-\lambda_0)$;
by flatness, exactness is preserved.

Now the minimal graded free resolution of $T/F^\perp$ is a quotient of $C$.
So $\beta_{5,6}(T/F^\perp) \leq \beta_{5,6}(C) = 135$.
\end{proof}

\begin{remark}
$C$ is not necessarily minimal.
It seems possible that in fact $C$ is never minimal, and $\beta_{5,6}(T/F^\perp) = 100$ for all $\mu,\lambda$ with $\mu \neq 0$.
\end{remark}

\begin{theorem}
$\rank(\dett_3) \geq 15$.
\end{theorem}
\begin{proof}
If $\field$ has characteristic $2$ or $3$ then $\dett_3$ cannot be expressed as a sum of cubes, so $\rank(\dett_3) = \infty$.
Indeed, in any cube $\ell^3$, the coefficient of any term of the form $x_i x_j x_k$ with $i,j,k$ distinct is divisible by $6$.

So we can assume $\field$ has characteristic $0$ or strictly greater than $3$.
Then all the above discussion about apolarity applies
(since it assumed that the field had characteristic $0$ or strictly greater than the degree of the polynomial under consideration).
If $\field$ is not closed under taking cube roots, then we pass to an extension of $\field$.
This can only decrease the Waring rank, so it is sufficient to get a lower bound over the extension field.
We may thus assume that $\field$ is closed under cube roots.

We have $\beta_{5,6}(T/(\dett_3 - \ell^3)^\perp) < 140$ for all linear forms $\ell$.
Therefore $\rank(\dett_3 - \ell^3) \geq 14$ for all $\ell$.
Hence $\rank(\dett_3) \geq 15$.
\end{proof}

\section{Remarks}

One might try to prove $\rank(\dett_3) \geq 16$ by considering $\dett_3 - \ell^3 - m^3$.
After a change of linear coordinates puts $\ell$ in normal form, a further linear change of coordinates to normalize $m$
must now fix both $\dett_3$ and $\ell$.
To simplify, we can assume $\ell = x_1 + x_5 + x_9$, with coefficient matrix given by the identity matrix.
Indeed, since $\det(Y) \aa \dett_d = 6 \neq 0$, there is no power sum decomposition of $\dett_d$ consisting entirely of terms
whose coefficient matrices are singular.
There is always at least one term with coefficient matrix of full rank, which we choose to be $\ell$.
Now the joint stabilizer of $\dett_3$ and the identity matrix is $\GL_3$ acting by conjugation.
Under this action we can put the coefficient matrix of $m$ in Jordan canonical form.
There are three cases: the coefficient matrix of $m$ may have Jordan type $(1,1,1)$, $(2,1)$, or $(3)$.
One may show that $\dett_3 - \ell^3$ is indecomposable and concise, hence $\dett_3 - \ell^3 - m^3$ is concise.
So as before, it is sufficient to show $\beta_{5,6} < 140$ for each $\mu \dett_3 - \lambda \ell^3 - \kappa m^3$,
where $m$ involves unknown (variable) eigenvalues.
Unfortunately at this point computer computations are prohibitive; we were not able to compute the needed free resolutions and graded Betti numbers.

The currently best known bounds in the case $d=4$ are $50 \leq \rank(\dett_4) \leq 160$.
The syzygies of $\dett_4^\perp$ and of $49$ or $50$ points in $\bbP^{15}$ do not seem to be suitable for this approach.

To try to improve the bound for $\rank(\per_3)$ by similarly considering $\per_3 - \ell^3$ involves a great number of cases.
The group of linear coordinate changes that leaves invariant the permanent is precisely the product of the torus acting by scaling rows and columns
(with total scaling $1$), the permutation group acting on rows and columns, and the transposition.
This group has many orbits, including the orbit of a matrix with all entries nonzero and independent.
Then at most we can ensure that the coefficient matrix of $\ell$ has $5$ entries of $1$ (the first entries in each row and column),
the other $4$ entries being independent variables.
We were not able to compute the apolar ideal, resolution, and graded Betti numbers in this case.

\section*{Acknowledgements}

The first author was partially supported by grant VR2013-4545 from the Swedish Research Council.
This work was supported by a grant from the Simons Foundation (\#354574, Zach Teitler).

This work began at the Fields Institute in Toronto
as part of the 2016 Fall Semester in \emph{Combinatorial Algebraic Geometry}.
Additional work took place at the
Stefan Banach Institute for Mathematics of the Polish Academy of Sciences (IMPAN) in Warsaw
during the September 2018 workshop on Varieties and Group Actions,
a part of the Simons Semester on \emph{Varieties: Arithmetic and Transformations} in Warsaw,
September 1 to December 1, 2018.
This work was partially supported by the grant 346300 for IMPAN
from the Simons Foundation and the matching 2015-2019 Polish MNiSW fund.

We are deeply grateful to the organizers and hosting institutions
of both of those special semesters,
in 2016 in Toronto and in 2018 in Warsaw.
We are grateful to Jaros{\l}aw Buczy{\'n}ski, Joachim Jelisiejew, J.M.\ Landsberg, Roy Skjelnes,
and the anonymous referee for many helpful comments and discussion.
%and to J.M.\ Landsberg for encouraging us in this work.

% for IMPAN: on arxiv, add 'report number' BCSim-2018-s09
% see: https://www.impan.pl/~vat/

% contact Fields semester organizers --- Greg Smith, et al
% ask them if they want any particular acknowledgement language,
% and also inform them, in case they're still writing any reports at this point
% (probably it's too late, but just in case)

\bigskip

\bibliographystyle{amsplain}       % Set the bibliography style to AMS
                                % alphabetized. (Can use ``amsalpha'' or
                                % ``abbrv''instead.) [amsplain, plain]
\renewcommand{\MR}[1]{}
\providecommand{\bysame}{\leavevmode\hbox to3em{\hrulefill}\thinspace}
\providecommand{\MR}{\relax\ifhmode\unskip\space\fi MR }
% \MRhref is called by the amsart/book/proc definition of \MR.
\providecommand{\MRhref}[2]{%
  \href{http://www.ams.org/mathscinet-getitem?mr=#1}{#2}
}
\providecommand{\href}[2]{#2}

\bigskip

\end{document}